\documentclass[12pt]{article}

\usepackage{amstext    }
\usepackage{amsthm    }
\usepackage{a4}
\usepackage[mathscr]{eucal}
\usepackage{mathrsfs}
\usepackage{hyperref}

\usepackage{amsmath}
\usepackage{amssymb}
\usepackage{amscd}

\numberwithin{equation}{section}

\newtheorem{theorem}{Theorem}[section]
\newtheorem{definition}[theorem]{Definition}
\newtheorem{proposition}[theorem]{Proposition}
\newtheorem{corollary}[theorem]{Corollary}
\newtheorem{lemma}[theorem]{Lemma}
\newtheorem{remark}[theorem]{Remark}

\newtheorem{example}[theorem]{Example}

\newcommand{\cali}[1]{\mathscr{#1}}

\newcommand{\supp}{{\rm supp}}
\newcommand{\const}{{\rm const}}
\newcommand{\dist}{{\rm dist}}

\newcommand{\dbar}{{\overline\partial}}
\newcommand{\ddbar}{{\partial\overline\partial}}

\newcommand{\rank}{{\rm  rank}}

\renewcommand{\Re}{{\rm Re}}

\newcommand{\Fc}{\cali{F}}

\newcommand{\C}{\mathbb{C}}

\renewcommand{\P}{\mathbb{P}}

\title{Levi Problem in Complex Manifolds}

\author{Nessim Sibony}

\date{{\it In memory of Raghavan Narasimhan}}

\begin{document}

\maketitle

\begin{abstract}
Let  $U $ be a  pseudoconvex open set in a complex manifold $M$. When  is $U$ a Stein manifold?
  There are classical counter examples due to Grauert, even when $U$ has real-analytic boundary or has strictly pseudoconvex points.
  We give new criteria for the Steinness of $U$ and we analyze the obstructions.  The main tool is 
  the notion of Levi-currents. They are  positive $\ddbar$-closed
 currents $T$ of bidimension $(1,1)$ and of mass $1$  directed by the directions where all continuous psh functions in $U$ have vanishing Levi-form. The extremal ones, are supported on the sets where all continuous psh functions are constant.
  We also  construct under geometric conditions, bounded strictly psh exhaustion functions, and hence we obtain Donnelly- Fefferman weights.
  To any infinitesimally homogeneous manifold, we associate a foliation. The dynamics of the foliation 
  determines the solution of the Levi-problem. 
  Some of the results can be extended to the context of pseudoconvexity with respect to a Pfaff-system.
    \end{abstract}

\noindent
{\bf Classification AMS 2010:} Primary: 32Q28, 32U10;  32U40; 32W05;  Secondary 37F75


 \section{Introduction} \label{intro}

 
 Let $(M,\omega)$ be  a  complex   Hermitian  manifold of dimension $n.$
 Let $U\Subset M$ be a relatively compact domain with smooth boundary. We can  assume that  
 $
 U:=\left\lbrace z\in U_1:\  r(z)<0  \right\rbrace,
 $
 where $r$ is a function of class $\mathcal C^\infty$ in a neighborhood $U_1$ of  $\overline{U},$
 such that   $dr$ is  non-vanishing on $\partial U.$ Recall that $U$ is  pseudoconvex if the Levi form
 of $r$ is nonnegative on the complex tangent  vectors  to $\partial U.$ More precisely,
 \begin{equation}\label{e:Levi}
 \langle i\ddbar r(z),it\wedge\bar{t} \rangle\geq 0\qquad\text{if}\qquad  \langle\partial r(z),t  \rangle=0.
 \end{equation}
 Condition \eqref{e:Levi} is  independent of the choice of $r.$
 
 The  Levi problem  is,  whether  a pseudoconvex domain is Stein, i.e., biholomorphic to a complex submanifold of $\C^N,$   see  \cite{Ho}.
 
 Grauert has   characterized  Stein manifolds as  follows.
 \begin{theorem} {\rm (\cite{G1})}
 A complex manifold $M$ is  Stein iff  there is a  strictly psh  exhaustion  function on $M.$
 \end{theorem}
 Narasimhan  has given  a  similar  characterization  for Stein  spaces \cite{N1}.

The Levi problem  admits a  positive   solution  in many cases, in particular, when $M$ is $\C^n$
or $\P^n.$ We refer to the  surveys by  Narasimhan \cite{N2},  Siu \cite{Siu}, PeternellÊ\cite {P}  and to the recent 
discussion by Ohsawa \cite{O}. See also  the book by H\"ormander \cite{Ho}. 

In  the  general case, Grauert has given two  remarkable examples.
Let $M:=\C^n/\Lambda$ be  a complex torus. Assume that $e_1:=(1,0,\ldots,0)$ is  the  first  vector 
in the lattice $\Lambda.$ Let $\pi:\  \C^n\to M$ denote the canonical projection.
Then  $U:=\pi(0<\Re z_1<1/2)$ is  pseudoconvex (Levi-flat) and $U$ is  not Stein.    Indeed,the compact set  $\pi(\Re z_1=1/4)$ is foliated by images of $\C,$ hence  holomorphic  functions  in $U$  which are  necessarily bounded
 on such images  are constant. 
 
   Hirschowitz has   analyzed such examples by introducing the  notion of 
 infinitesimally  homogeneous manifolds. A manifold $M$ is infinitesimally  homogeneous if the  global
 holomorphic vector  fields  generate the  tangent space at  every point of $M.$ He then  showed \cite{H1,H2}
 \begin{theorem}{\rm (Hirschowitz) }  Let $U$
 be a  domain in an  infinitesimally homogeneous manifold. Assume $U$ satisfies the Kontinuit\"atssatz.  Then $U$  admits a continuous 
 psh  exhaustion function. If moreover $U$ does not contain a holomorphic image of  $\C,$ which is  relatively compact in $U,$ then  $U$ is  Stein.
 \end{theorem}

 A  second  example of Grauert \cite{G2} shows that the  boundary  of $U$  can be   strictly pseudoconvex
except on a small set and still, $U$ is  not Stein. In the present article  we analyze the  obstructions of being Stein for pseudoconvex  domains  with   smooth  boundary. Our main tool is the  notion of Levi currents.
With  the  previous notations,  a  positive  current $T$ of  bidimension $(1,1)$ in $M,$ supported 
on $\partial U$
is  a Levi current if it satisfies the  following  Pfaff  system

\begin{equation}\label{e:Pfaff}
T\wedge \partial r=0,\qquad T\wedge \ddbar r=0,\qquad i\ddbar T=0,\qquad \langle T,\omega\rangle
=1.\end{equation} 
   Observe  that the support of the Levi current is  very restricted, and that  it is directed by the null space of the Levi form. We then  obtain the   following result.
   
   \begin{theorem}\label{T:main_1}
   Let $U\Subset M$ be   a  pseudoconvex  domain with smooth  boundary.  If $\partial U$ has  no  Levi current, then $U$ is  a  modification of  a Stein  manifold. Moreover, there is   a smooth function $v,$
   such that if  $\rho:=re^{-v},$ there is  $\eta>0$ such that the function $\hat\rho:=-(-\rho)^\eta$
   satisfies $i\ddbar\hat\rho\geq  c|\hat\rho|\omega$ on  $U\setminus K,$ where  $K$ is compact.
   \end{theorem}

   Clearly,  when $M=\C^n,$ or a Stein manifold,  Levi  currents  do not exist. Indeed,  positive  currents  with compact  support, satisfying  $i\ddbar T=0,$
   are necessarily $0.$ So the  last part of the  above theorem  is an extension of the Diederich-Forn{\ae}ss theorem \cite{DF1} which considers the case where $M=\C^n.$ A crucial point in  their proof is that, 
   for a  pseudoconvex  domain $U$ in $\C^n,$  the  function $-\log\dist(\cdot,\C^n\setminus \overline U)$
   is  psh. This  tool is  not available  here. 
   
   A similar result was  proved when $M=\P^n$ or more generally
    for manifolds of  positive holomorphic sectional  curvature by Ohsawa and the author in \cite{OS}. It uses some  geometric
    inequalities  satisfied  by the  distance  to the  boundary due to Takeuchi and Elencwajg \cite{T,E}.
    
    The  interest  of constructing bounded exhaustion functions  satisfying the  above   estimates  is that the  function $\psi:=-\log(-\hat\rho)$ satisfies the Donnelly-Fefferman condition  and  is proper and  hence
one can apply their theorem  \cite{DoF}, see also \cite{B}.

When the   domain  $U$ has  real analytic  boundary  the   non-existence of Levi currents  is  equivalent to the   non-existence  of a germ of   holomorphic  curve on the   boundary of $U.$ This  uses  results
from  \cite{DF2}.

  In Section  2, after  proving the  above  results, we  address the   question of finding  bounded strictly 
  psh exhaustion function $\hat\rho,$  such that $\psi:=-\log(-\hat\rho)$ satisfies the Donnelly-Fefferman condition. We show in particular the following result.

  \begin{theorem}\label{T:main_2}
  Let $U\Subset M$  be a pseudoconvex domain with   smooth  boundary.  Assume  that there is a compact set $
  K\Subset U$ and  a  bounded  function  $v$ on $U\setminus K,$ such that $i\ddbar v\geq \omega$ on $U\setminus K.$ Then $U$ is  a  modification  of a Stein  manifold, and  admits  a bounded exhaustion function which is  strictly  psh  out of a compact set.   
 \end{theorem}
 
 In Section \ref{S:Steiness} we give   a criterion for  Steiness.
 In  Section \ref{S:Levi} we introduce the notion of  Levi currents on an arbitrary complex  manifold
 (not just on the   boundary  of a pseudoconvex  domain).  This  permits to prove  the following.
 \begin{theorem}\label{T:main_3}
 Let  $U\Subset M$ be a  pseudoconvex  domain with smooth boundary. Assume  it  admits a  continuous psh exhaustion function $\varphi.$ Assume  $U$  is  not a  modification of a Stein manifold and that it contains at most finitely many compact varieties of positive dimension. Then  there is  a number $t_0,$
 such that for every  $t>t_0$  the  level set $\{\varphi=t\}$ has a  Levi-current $T_t.$
 In particular,  each $T_t$ is  a positive $\ddbar$-closed  current of mass one  with compact support.
  \end{theorem}
We also address briefly the general question of the existence of bounded strictly psh functions, through the notion of Liouville currents.  
  
  In Section \ref{S:manifolds_vector_fields}, we  show that  if $M$ is  infinitesimally  homogeneous and $U$ is not  Stein, then  $U$ is  foliated  by complex  manifolds of fixed dimension $d>0,$
  and   the  closure of each leave is compact in  $U.$
  
  In Section \ref{S:Levi-Pfaff}  we give a  foliated version of the  above results.  More  precisely,  we
  develop the notion  of pseudoconvexity  with  respect to a  Pfaff system.

\section{Bounded psh exhaustion functions}\label{S:bounded_psh}


The  proof of Theorem \ref{T:main_1} is  based on the  following proposition.
\begin{proposition}\label{P:2.1}
Let $U\Subset  M$ be  a  pseudoconvex domain  with smooth  boundary.  There is no  Levi current
on $\partial U$  iff  there is   a smooth  strictly psh  function  $u$ in a  neighborhood of  $\partial U.$
\end{proposition}
\proof
If $u$  is  a  strictly psh  function in a  neighborhood  of $\partial U$ and  $T$ is  a  positive  current 
supported on  $\partial U,$ then
$$
\langle T,i\ddbar u  \rangle=\langle i\ddbar T,u   \rangle.
$$
So if $i\ddbar T=0,$ we get  that  $T=0.$  Hence  there is  no  Levi current.

We next show that  any positive $\ddbar$-closed current $T$ of mass one supported on $\partial U$
is a  Levi current. Since $T$ is $\ddbar$-closed,  then  
  $\langle T, i\ddbar r^2\rangle=0.$ Expanding and using that it is supported on $\{r=0\},$we get,  
  $T\wedge i\partial r\wedge\overline{\partial} r=0.$
Therefore,  $T\wedge \partial r=0.$

Let $\chi$ be  a  smooth non-negative  function   with  compact  support. Using that
$T\wedge \partial r=0,$ we get that:
$$
0=\langle T,i\ddbar (\chi r)  \rangle=\langle T, \chi i\ddbar r   \rangle.
$$ 
But $\partial U$ is  pseudoconvex, i.e., $\langle i\ddbar r, it\wedge \bar{t}   \rangle\geq 0$ when
$\langle \partial r,t\rangle=0.$
The  current $T$ is   directed  by   the complex  tangent  current  space to $\partial U$
because  $T\wedge \partial r=0.$
It follows that $T\wedge\chi i\ddbar r=0$ for  an  arbitrary  $\chi.$ Hence, $T$ is a Levi current
on $\partial U.$ So  it is   enough  to show that if there is  no $\ddbar$-closed  positive  current of mass one supported on $\partial U,$ there is   a  smooth  strictly  psh  function  in  a neighborhood of
$\partial U.$

Let  
$$
\mathcal C:=\left\lbrace T:\    T\geq 0\quad \text{bidimension}\quad (1,1)\quad\text{supported on }\quad
\partial U,\quad \langle T,\omega\rangle=1  \right\rbrace,
$$
and
$$ Y:=\left\lbrace i\ddbar u,\  u\ \text{test smooth function on M}   \right\rbrace^\perp.$$ 

The space $Y$ is  the  space  of   the $i\ddbar$-closed currents on $M.$
We have  assumed that $\mathcal C\cap Y$ is  empty. The convex  compact $\mathcal C$ is  in   the  dual of  a reflexive  space. The  Hahn-Banach theorem implies that $\mathcal C$ and $Y$  are  strongly  separated.
Hence, there is  $\delta>0$ and a  test function $u,$ such that $\langle  i\ddbar u,T\rangle  \geq \delta,$
for every  $T$ in $\mathcal C.$ So the function  $u,$ is  strictly psh  at all points of $\partial U,$  and hence   in a
neighborhood of  $\partial U.$ Similar use of  Hahn-Banach theorem occurs in  \cite{S1, Su}
\endproof

Since   we have a  strongly psh function  in  a  neighborhood of $\partial U,$  Theorem  \ref{T:main_1}
 will be  a  consequence of the  following theorem.
 
 \begin{theorem}\label{T:2.2}
 Let $K$ be  a  compact   set in $U.$  Assume  there is  	  function  $v$  in $U\setminus K$ such that
 one of the  following   conditions  is  satisfied:
 \begin{enumerate}
 \item[(i)] $v$ is  bounded and  $i\ddbar v\geq \omega;$
 \item[(ii)]  $v\geq 0,$ $i\partial v\wedge \overline{\partial}v\leq i\ddbar v,$ and 
 $i\ddbar v\geq \omega.$
 \end{enumerate}
 Then there is a  psh exhaustion function $\hat\rho,$ vanishing  on $\partial U$ and  such that
 on $U\setminus K,$
 $i\ddbar  \hat\rho \geq  |\hat\rho|\omega.$ 
 \end{theorem}
 \proof
 We know that  for $z\in\partial U,$ $\langle i\ddbar r(z),it\wedge \bar{t}   \rangle\geq 0,$ when
 $\langle \partial r(z),t\rangle=0.$ It  follows that there is a  constant $C,$  such that for $z\in\partial U$ and $t$ arbitrary  in the  tangent  space $T^{(1,0)}_z(M),$
 $$
 \langle i\ddbar r(z),it\wedge \bar{t}   \rangle\geq-C|\langle \partial r(z),t\rangle| |t|.
 $$
  
 Choose    a  small neighborhood $V$ of  $\partial U$ such that every  point $z$ in $V$  projects  to  a  point $z_1\in\partial U.$  Then $\partial r(z)=\partial r(z_1)+O(r)$ Hence,
 \begin{equation}\label{e:2.1}
 \begin{split}
 \langle i\ddbar r(z),it\wedge \bar{t}   \rangle&\geq  \langle i\ddbar r(z_1),it\wedge \bar{t}   \rangle
 -C_0|r(z)||t|^2\\
 &\geq-C_1 |\langle \partial r(z_1),t\rangle||t|-C_1|r(z)||t|^2\\
 &\geq-C|\langle \partial r(z),t\rangle||t|-C|r(z)||t|^2.
 \end{split}
 \end{equation}
 Define  $\rho:=re^{-Av}$ and  $\hat\rho:=-(-\rho)^\eta,$  we  will choose  $A$ and $\eta$ later.
 Observe  first that by Richberg's theorem \cite{R}, we can  assume that  $v$ is   smooth.  We have
 $$
 i\ddbar \hat\rho=\eta|r|^{\eta-2}e^{-A\eta v}[D(t)].
 $$
 To get that $i\ddbar \hat\rho\gtrsim|r|^2\omega,$ we  need to show that
 \begin{equation}\label{e:star}
 [D(t)]\gtrsim |r|^2\omega.
 \end{equation}
Let $\mathcal Lv$ denote $\langle  i\ddbar v(z),it\wedge\bar{t}\rangle.$  We then have 
\begin{eqnarray*}
D(t)&=& Ar^2(\mathcal Lv -\eta A|\langle \partial v,t\rangle|^2)+|r|(\mathcal Lr-2\eta\Re \langle \partial r,t\rangle
\overline{ \langle \partial v,t\rangle}\big)\\
&+& (1-\eta) |\langle \partial r,t\rangle|^2.
\end{eqnarray*}
We also have
$$
2\eta|r| |\Re\langle \partial r,t\rangle \overline{\langle\partial v,t \rangle|}\leq r^2
|\langle\partial v,t \rangle|^2+\eta^2 |\langle\partial r,t \rangle|^2.
$$
So using relation \eqref{e:2.1} we get
\begin{eqnarray*}
D(t) &\geq & Ar^2\big (\mathcal L v -(\eta A+A^{-1}) |\langle\partial v,t \rangle|^2  \big)
+(1-\eta-\eta^2)|\langle\partial r,t \rangle|^2\\
&-& C|r||\langle\partial r,t \rangle||t|-Cr^2|t|^2.
\end{eqnarray*}
Hence,
\begin{eqnarray*}
D(t) &\geq & Ar^2(\mathcal Lv -{C\over A}|t|^2- (\eta A+ {1\over A})  |\langle\partial v,t \rangle|^2
- {1\over A\sqrt{\eta}}|t|^2 \Big)\\
&+& (1-\eta-\eta^2 -C\sqrt{\eta}) |\langle\partial r,t \rangle|^2.  
\end{eqnarray*}
Since $i\ddbar v\geq \omega,$ and $i\partial v\wedge \bar{\partial} v\leq i\ddbar v,$ it suffices to take
$A\simeq  {1\over 2\sqrt{\eta}}$ and $\eta$ small  enough. If $v$ is  bounded, we can assume $v\geq 0$
and  replace $v$ by $Cv^2,$ then condition (ii) is  satisfied. 
\endproof
\begin{remark}\rm
(i) In particular, we  obtain that $U$ is  a  modification of a Stein space.\\
(ii) Observe  that when $v$ extends  smoothly to $\partial U,$ as in Theorem \ref{T:main_1}, then
$\hat\rho$ is H\"older  continuous.\\
(iii)  The  conditions on $v$  are of the type  required for the Donnelly-Fefferman weights, except 
we do not ask for completeness i.e. that $v\to\infty$ when we approach $\partial U.$
\end{remark}
\begin{example}
Let $(M,\omega)$  be  a compact K\"ahler  manifold. Let $T$ be  a positive closed current of bidegree $(1,1),$ cohomologous  to $\omega.$ Write $T-\omega=i\ddbar v,$ we can assume $v\leq 0.$  
Assume  $U$ is  pseudoconvex disjoint from the support of $T.$  Then on $U,$  we have
$\omega=i\ddbar(-v).$ The hypothesis of Theorem \ref{T:2.2} is  satisfied if $T$ admits locally  bounded
potentials.  Otherwise the  hypothesis of Theorem \ref{T:3.1} below is  satisfied.
\end{example}

\begin{theorem}\label{T:2.3}
Let $U\Subset M$ be  smooth  pseudoconvex  with real analytic  boundary. Then  $\partial U$ has no  Levi current iff it contains no germ of holomorphic  curve. 
\end{theorem}
\proof

Suppose  $\partial U$ has  no Levi current.  Then there is a  smooth  strictly psh function $v$ near $\partial U.$  Let $W$
denote the  union of non-trivial  germs of holomorphic discs on $\partial U$ and  suppose $W$ is  nonempty.
Consider the  closure $\overline{W}$  and let $p\in\overline{W}$ where the  function $v$ reaches  it maximum on $\overline{W}.$
According to the proof  of  Theorem  4  of \cite{DF2}  there is  a  point $q$ close to $p$ and  a  nontrivial  subvariety $V$ through $q,$
in a  polydisc  centered at $q$ of radius $\delta.$ Moreover, we can choose $q$  arbitrarily close to $p,$ without  changing $\delta.$ Since $v$  is   strictly psh,
we can  assume  that the maximum at $p$  is  reached at  an interior point  of $V.$ A  contradiction. So  $W$ is  empty.

Assume now that $\partial U$ does  not have  a non-trivial germ of holomorphic disc. It follows from Theorem 3 in  \cite{DF2} that  the  holomorphic   dimension of any  real analytic
 submanifold $N$ is  zero. This  means that  for every $z\in N,$  $T^{(1,0)}_z(N)$ intersects 
 $$
 M_z:=\left\lbrace  t:\   t\in T^{(1,0)}_z(\partial U),\   \langle i\ddbar r(z),it\wedge\bar{t}\rangle=0   \right\rbrace
 $$
only at $0.$  So for each  non-zero vector $t\in T^{(1,0)}_z(N)$, $\langle i\ddbar r(z),it\wedge\bar{t}\rangle>0.$
The authors in \cite{DF2} state  and prove their  theorem in $\C^n$ but this part  of the  argument  is  of local nature.
Let $N_0$  denote the real analytic set where $\dim M_z>0.$ Then
$N_0\subset \bigcup_{k=1}^r N_k,$  where $N_k$   is a closed submanifold in $\partial U\setminus \bigcup_{j=1}^{k-1} N_j,$ 
moreover
$\langle i\ddbar r(z),it\wedge\bar{t}\rangle>0$ for a non-zero   $t\in T^{(1,0)}_z(N_k).$ This follows  from the Lojasiewicz  stratification of real analytic sets  and  from the 
above  statment, see \cite{DF2}. The Levi current is  a  priori supported on $N_0.$

Let $\rho_j$  be  a  defining function of $N_j$  and  let $\chi$  be a  cutoff function. If we expand $\langle T,i\ddbar (\chi\rho_j^2)\rangle=0,$
 we get that $T\wedge \partial \rho_j=0.$

Writting   that  $\langle T,i\ddbar (\chi\rho_j)\rangle=0,$ we get also that $T\wedge i\ddbar \rho_j=0.$ The non-degeneracy of $i\ddbar \rho_j$ on $M_z$ implies that
$T=0.$
\endproof

 We recall  the following form  of the  following  Donnelly-Fefferman theorem, see \cite{DoF} and \cite{B}.
 
 \begin{theorem}\label{T:2.4}
  Let $N$  be a  complex  manifold of dimension $n.$  Let $\Omega:=i\ddbar\varphi$  be  a complete K\"ahler metric on $N.$
  Assume there is $C_0>0$  such that
  $i\partial \varphi\wedge \overline\partial{\varphi}\leq C_0 i\ddbar\varphi.$
  
  Assume  $p+q\not=n.$
Then, for any $(p,q)$-form $f$  in $L^2$ with  $\overline\partial f=0,$ there  is  a  solution $u,$ to  the  equation  
   $\overline\partial u=f$ with
   $$
  \|  u\|^2_\Omega\leq C\|f\|^2_\Omega.
   $$
  
  \end{theorem}
The condition on $\varphi$  means just that  $|d\varphi|_\Omega$ is  bounded. The  completeness means that  $\varphi(z)\to\infty$ when  $z\to\infty$ on $N.$
The following proposition permits to apply the above  theorem  to the  pseudoconvex domains  considered previously. We just have to assume that $U$ does not contain analytic varieties of
positive dimension.

\begin{proposition}
 Let $U\Subset M$ be  a  pseudoconvex domain  with a negative exhaustion  function $\hat\rho,$  satisfying
 $$
 i\ddbar\hat\rho \gtrsim |\hat\rho|\omega.
 $$
 Let $\varphi:=-\log(-\hat\rho).$  Then  the metric   $\Omega:= i\ddbar\varphi$ is complete  and  $|d\varphi|_\Omega$ is  bounded. So Theorem \ref{T:2.4} applies.
 \end{proposition}
\proof
We have
$$
i\ddbar\varphi={i\ddbar\hat\rho \over|\hat\rho|}+  {i\partial\hat\rho \wedge \overline\partial\hat\rho\over \hat\rho^2}
={i\ddbar\hat\rho \over |\hat\rho|}+i\partial\varphi \wedge \overline\partial\varphi\geq   c\omega+i\partial\varphi \wedge \overline\partial\varphi.
$$
Moreover,   since $\hat\rho\to 0$ when  $z\to\partial U,$ the  metric $\Omega$  is complete.
\endproof

\begin{remark}
 In Theorem  \ref{T:2.2}, we  start  with a potential $v$ satisfying  $i\partial v\wedge\overline\partial v\leq  i\ddbar v,$ but the metric is  not necessarily complete. We end up
 with a  complete one associated to $\varphi:=-\log(-\hat\rho).$
\end{remark}

\endproof

\section{A condition for Steiness of a pseudoconvex domain}\label{S:Steiness}

As  recalled in the  introduction, according to Grauert's theorem, to prove that  a  pseudoconvex domain is  Stein, one  should   construct  a  strictly psh  exhaustion function.
Here  we give  a  quite  weak assumption in order to construct   such an exhaustion. 

\begin{theorem}\label{T:3.1}
 Let $U\Subset M$ be  a  pseudoconvex domain with smooth  boundary. Assume there is a  neighborhood $V$ of $\partial U,$
 and  a  function $v$ on $U\cap V$ such that the following conditions are satisfied:
 \begin{enumerate}
  \item[ (i)] $i\ddbar v\geq \omega;$
  \item[(ii)] If $r$ denotes a  defining   function  for $\partial U,$ then for any $\epsilon>0,$
  $v>\epsilon\log{|r|}$ when $r\to 0.$
 \end{enumerate}
Then $U$ admits a  bounded exhaustion function, with  all level sets  strictly pseudoconvex. Moreover,  $U$ is  a modification 
of a Stein space.
\end{theorem}
When   $v$ is  defined near $\partial U$ and  satisfies $i\ddbar v\geq \omega,$ Elencwajg \cite{E}  showed that  $U$ is  a modification 
of a Stein space (see also \cite{Siu}).
\proof
Define  $\sigma:=re^{-Av}.$ According to  Richberg's approximation theorem \cite{R} we can   assume that $v$ is  smooth. Condition  (ii) implies that for every $A>0,$
$\sigma$ is an exhaustion function. We have
$\overline\partial \sigma=e^{-Av}(\overline\partial r-Ar\overline\partial v)$ and 
$$
i\ddbar \sigma=e^{-Av}\Big( i\ddbar r- 2A\Re(i\partial r\wedge \overline\partial v)-Ari\ddbar v+A^2r\partial v\wedge \overline\partial v \Big).
$$
We are going to check that the level sets  of $\sigma$  are strictly pseudoconvex. If $t$ is   a $(1,0)$ tangent vector  to a level set, then 
$\langle \partial \sigma,t\rangle=0$ i.e.  $\langle \partial r,t\rangle=Ar\langle \partial v,t\rangle.$
So 
\begin{eqnarray*}
\langle  i\ddbar \sigma, it\wedge \overline{t} \rangle&=&e^{-Av}\big ( \langle i\ddbar r, it\wedge \overline{t} \rangle+2A^2|r||\langle  \partial v,t\rangle|^2\\
&+&A^2r |\langle  \partial v,t\rangle|^2+A|r|\langle i\ddbar v,  it\wedge \overline{t} \rangle\big). 
\end{eqnarray*}
We also have  near  $\partial U$ that

$$
\langle  i\ddbar r(z), it\wedge \overline{t} \rangle
\geq  -C\big ( |\langle  \partial r(z),t\rangle||t|+|r||t|^2  \big).
$$
So if $A>C,$ and $A$ is large enough.
\begin{eqnarray*}
e^{Av} \langle i\ddbar \sigma, it\wedge \overline{t} \rangle &=& -CA |r|\ |\langle  \partial v,t\rangle|^2 -2C|r| |t|^2\\
&+& A^2|r||\langle  \partial v,t\rangle|^2 +A|r| \langle i\ddbar v, it\wedge \overline{t} \rangle\\
&\geq & |r| (  A\langle  i\ddbar v,it\wedge \bar{t}\rangle-2C|t|^2 )>0. 
\end{eqnarray*}

It follows that there is  a  function  $C(\sigma)$ such  that
$$
i\ddbar \sigma \geq   -{C(\sigma)\over 2}i\partial\sigma\wedge \overline\partial\sigma\qquad\text{near}\qquad \partial U.
$$
Let  $\kappa(t):=\int_{-1}^t C(s)ds,$ and $\chi(t):=\int_{-1}^t e^{\kappa(s)}ds.$
Then $\chi''-C(s)\chi'(s)=0.$  Define $\rho:=\chi(\sigma).$ Then $\rho$ is a  strictly psh  exhaustion  function. Indeed,
\begin{multline*}
i\ddbar \rho=\chi'(\sigma)i\ddbar\sigma+\chi''(\sigma)i\partial\sigma\wedge\overline\partial\sigma\\
> \big( {-C\over 2} \chi'(\sigma) +\chi''(\sigma) \big)i\partial\sigma\wedge\overline\partial\sigma\geq {C\over 2}\chi'(\sigma)i\partial\sigma\wedge\overline\partial\sigma.
\end{multline*}
On the other hand, when $\langle\partial\sigma,t \rangle=0,$  we also  have  strict positivity.  So $\rho$ is a  strictly  psh exhaustion  function.
\endproof


\section{An obstruction to Steiness: Levi currents}\label{S:Levi}


Let  $U\Subset M$
be  a locally  Stein domain in a complex Hermitian   manifold  $(M,\omega).$ It is  not clear  whether   there  are  non-constant  psh functions in $U.$ 

When $U$ admits a continuous psh exhaustion function, the domain $U$ may  not have   strictly  psh   functions and hence is not  necessarily Stein, this is the case in Grauert examples \cite{G1,G2}, or   in the   families  described
by Ohsawa  \cite{O}.
 When $M$ is  infinitesimally homogeneous manifold \cite{H1,H2},  the domain $U$ has a psh exhaustion function, but may  not have   strictly  psh   functions and hence is not  necessarily Stein.
  
In this  section we  want to discuss an obstruction  to  Steiness  given by   Levi currents  with  compact support in $U.$ In order  to introduce the notion
we need to define  $T\wedge i\ddbar v,$ when  $T$ is a  positive current $\ddbar$-closed  and $v$ is a  continuous  psh function.
We recall   few  results  from  \cite{DS1}.

Let $T$ be  a positive  current of bidegree $(p,p).$   Assume that $i\ddbar T$ is  of order $0.$
When $T$ is a  current of order $0,$ the mass of $T$ on a compact $K$  is  denoted by $\| T\|_K.$
When $T$ is  positive, and $M$ is  of dimension $n,$  $\| T\|_K$ is  equivalent to  $\big|\int_K T\wedge \omega^{n-p}\big|.$

 When $u$ is  a  smooth psh function on an open  set $V\subset M$
we have
\begin{equation}\label{e:4.1}
i\ddbar u\wedge T:=u(i\ddbar T) -i\ddbar(uT)+i\partial(\dbar u\wedge T)-i\dbar(\partial u\wedge T).
\end{equation}
The following estimate is proved in  \cite{DS1}. Let $L\Subset K$  be two  compact sets in $V.$  Assume  $T$ is  positive 
and $i\ddbar T$ is  of order $0.$ Then there is  a  constant $C_{K,L}>0$  such that   for every  smooth
bounded psh function $u$ on $V,$ we have
\begin{equation}\label{e:4.2}
\int_L i\partial u\wedge \dbar u\wedge T\wedge \omega^{n-p-1}\leq  C_{K,L}\|u\|^2_{L^\infty(K)}\big (\|T\|_K+\|i\ddbar T\|_K\big)
\end{equation}
and
\begin{equation}\label{e:4.3}
\| i \ddbar u\wedge T\|_L\leq  C_{K,L}\|u\|_{L^\infty(K)}\big (\|T\|_K+\|i\ddbar T\|_K\big).
\end{equation}
This permits to extend relation  \eqref{e:4.1} to $u$ psh  and continuous. Moreover, when $u_n$ converges locally uniformly
to $u$ then 
$$
I_{n,m}:=\int_L i\partial (u_n-u_m)\wedge \dbar (u_n-u_m)\wedge T\wedge \omega^{k-p-1}
$$
converges to $0.$ It is  enough to prove that  for a ball  $B$  and  to assume $u_n$ and $u_m$  coincide  near  the boundary  of $B$, see \cite{DS1}. So
$$
I_{n,m}={-1\over 2}\int  (u_n-u_m) i\ddbar (u_n-u_m)\wedge T+{1\over 2}\int  i\ddbar (u_n-u_m)^2\wedge T.
$$
Hence,  
\begin{eqnarray*}
2I_{n,m}&\leq&  \|u_n-u_m\|^2_B\| i\ddbar T\|_B +\int_B|u_n-u_m|i\ddbar u_n\wedge T +\int_B|u_n-u_m|i\ddbar u_m \wedge T\\
&+&\int_B|u_n-u_m|i\ddbar u_n \wedge T.
\end{eqnarray*}
The convergence  follows   using \eqref{e:4.3}.

Estimate \eqref{e:4.2}  permits also to define $\partial u\wedge T.$ Then  $\partial u_n\wedge T\to  \partial u\wedge T$, as currents of order $0,$  if $u_n$ converges to $u$ uniformly on compact sets.

\begin{definition}
 A Levi current    in $V$  is    a nonzero positive current of bidimension $(1,1),$ such that  $i\ddbar T=0$  and $T\wedge i\ddbar v=0$
 for every continuous psh function in $V.$
 
 A Liouville current    in $V$  is    a nonzero positive current of bidimension $(1,1),$ such that  $i\ddbar T=0$  and $T\wedge i\ddbar v=0$
 for every  bounded continuous psh function in $V.$
 \end{definition}
 Observe   that this  implies that for a Levi current  $T$ (resp. a Liouville current) we have that
 $T\wedge  \dbar v=0$ for  every  continuous  psh function $v,$ (resp.for a bounded continuous psh function $v$).
 \begin{proposition}
  Let $K$ be a compact set  in $V.$  If $S$ is a positive current of bidimension $(1,1),$ supported in  $K,$ such that  $i\ddbar S=0,$  then $S$ is a Levi current. The convex set  $\mathcal L(K)$ of Levi currents of mass $1,$ supported in  $K$ is compact.
  
  Let $T$  be a Levi current in $V$.  Let $v$ be a non-negative continuous psh function, then the current $vT,$ is a Levi current. If $T$ is an extremal Levi current in $V, $ then continuous psh functions in $V,$ are constant on $H:=\supp(T).$
  
   A similar statement holds for Liouville currents in $V.$
 \end{proposition}
\proof
Suppose $S$ is a positive current of bidimension $(1,1),$ $\ddbar$-closed and supported on $K.$
We observe first that for $u$ continuous and psh in a neighborhood of $K,$ $\dbar u\wedge S,$ is well defined and of order $0.$Then 
we apply  \eqref{e:4.1} to the function $1.$ This shows that $\ddbar u\wedge S= 0.$ So $S$ is a Levi current. It follows that  $\mathcal L(K)$  is compact.

Assume $T$ is a Levi current in $V$. Let $v$ be a continuous psh function in $V.$ Let $h$ be a convex strictly increasing function. Since $\ddbar h(v)\wedge T=0,$ we get that $\dbar v\wedge T=0.$  We then apply formula \eqref{e:4.1}  and get that $-\ddbar (vT)=\ddbar v\wedge T=0.$

Assume $T$ is extremal. Let  $u$ be   continuous   psh  in $V.$  Suppose
 $(u<0)$ and  $(u>0)$   are two
nonempty open  sets   in $H.$  Let $\chi$  be a convex  increasing  function   vanishing for $t<0$
and  strictly increasing for $t>0.$ Then  the current $S:=\chi(u)T$ is  a  Levi current as we have seen. This contradicts the extremality of $T.$ 

The proof for Liouville currents is similar.

\endproof

\begin{theorem}\label{T:4.3}
 Let $N$ be  a complex   manifold   with a psh  continuous   exhaustion $\varphi.$ Then $N$ is Stein iff  there is  no Levi  current
 with   compact support in $N.$
 
  If  $N$ admits a  bounded  continuous  psh  exhaustion function, then  it   admits  a  bounded   strictly 
  psh  exhaustion    function  iff there is  no Levi   current   with compact   suport in $N.$
\end{theorem}

\proof
If there is  a  strictly psh,  continuous function $v$  in $N$ and $T$
is   a positive  current   such that    $T\wedge \ddbar v=0,$ then $T=0.$
We have to prove the converse. We show that  if there is   no Levi current on a compact set $K,$
then there  is a smooth  function $v_K, $ strictly psh in a neighborhood of $K.$

Suppose $S$ is a positive current of bidimension $(1,1)$ $\ddbar$-closed and supported on $K.$
As we have seen,  for every $u$ continuous psh near $K,$
in particular, for $u$ continuous psh on $N,$ we have that $S\wedge i\ddbar u=0.$ So $S$ is a Levi current. Hence $S=0.$ The duality argument  used in the proof of Proposition \ref{P:2.1} implies the  existence of $v_K$ smooth and  strictly psh near $K.$

For a compact $K$ let $\widehat{K}$  denote the  hull  with respect to  continuous psh functions. Since there is a psh continuous  exhaustion function  we can choose $K_n\nearrow N,$ $K_{n+1}\Subset K_n,$
and $K_n=\widehat{K}_n.$ Let $v_n$ be  a continuous  function strictly  psh near  $K_{n+1}.$ Let $\chi_n$ be  a convex  increasing function such that
$$
\chi_n(\varphi)<\inf_{K_n} v_n\quad\text{on}\quad K_{n-1}.\quad \text{and} \quad \chi_n(\varphi)>\sup v_n\quad\text{near}\quad \partial K_n.
$$    
Define 
$$
u_n:=\sup(\chi_n(\varphi),v_n).
$$
Then $u_n$ is psh continuous  on $N$ and strictly psh near $K_{n-1}.$ Moreover, it is an exhaustion.

If we choose  $0<\epsilon_n< {1\over 2^n} \|u_n\|^{-1}_{K_n},$ then  the function $u:=\sum_n\epsilon_nu_n$
is strictly psh function. So the  function   $u+\chi(\varphi)$ is  strictly  psh and an exhaustion if $\chi$  is  convex increasing  fast  enough. 

Suppose  now that $\varphi<0$  psh and $\varphi\to 0$
when  $z\to\infty$ on $N.$ We consider $u_n$ as  above  and  define
$w_n:=\epsilon_n(u_n-c_n)$  with $c_n:=\lim_{z\to\infty} u_n(z).$ It is  clear that $c_n$ is   constant.
If $\epsilon_n$ is  small enough, $w:=\sum w_n$ is a  bounded strictly psh  exhaustion.
 \endproof

\proof[Proof of Theorem  \ref{T:main_3}]
We will  need the following theorem of Grauert.  If $N$ is a complex  manifold
with a  continuous psh exhaustion $\varphi,$ such that  $\varphi$ is  strictly psh out of  a  compact set
$K$ of $N,$ then   $N$ is a  proper  modification  of a Stein space. More precisely,  one can blow down   analytic sets
 in $N$ to points and  get holomorphic  convexity  for compact sets  in the  blow down.
 
 Suppose    $N$ is  not a modification  of a Stein space  and that  for a sequence $t_n\to\infty,$ there is no Levi current on $(\varphi=t_n).$ We have seen that this implies the  existence  of a smooth strictly psh
  function  near  $(\varphi=t_n).$ Using the construction in the previous  theorem, there is a continuous psh
  exhaustion $\psi,$ strictly psh  in a neighborhood of each   $(\varphi=t_n).$ So  $(\varphi< t_n)$ is    a  modification of a Stein space. In particular the  compact analytic sets $A_j$ are necessarily in   $(\varphi=s_j),$  for some  
  $s_j.$  From the finiteness 
  assumption  the $A_j$ cannot  accumulate  near the  boundary. So  the $s_j$ are uniformly bounded. Hence, there is  $t_1$   such that  for $t>t_1,$ there is  no  Levi current with compact support on $t>t_1.$  Here again, we use Grauert's Theorem.   The argument  in Theorem \ref{T:4.3}  shows that  we can  construct  a  psh  exhaustion  function  strictly psh on 
 $(\varphi>t_1).$ Hence   $U$ is a  modification  of a  Stein space. A contradiction. So, for $t$ large enough there is  a Levi current on $(\varphi=t).$ 
 \endproof
The following proposition describes  the function theory near the support of a positive $\ddbar$-closed current in $\P^k.$

  \begin{proposition}
  Let $T$ be an extremal positive  current of bidimension $(1,1),$ $\ddbar$-closed in $\P^k$ with support $K.$ Then there exists a fundamental sequence  of open neighbohoods $(U_n)$ of $K$, such that every psh function in $U_n$ is constant.
  \end{proposition}
  \proof
  Let $u$ be a psh function in a neighborhood $V$ of $K.$ Since $\P^k$ is homogeneous, we can assume that $u$ is smooth and satisfies $0<u<1.$ As we have seen $T$ is a Levi current, hence $uT$
  is $\ddbar$-closed. The extremality of $T$ implies that $u$ is constant on $K.$ We can consider, the images of $T$ by automorphisms of $\P^k,$ close to identity with a fixed point in $K.$  The function $u$
  has to be constant on the images of $K.$ The theorem follows.
 \endproof

 \begin{remark}\rm
 1) Grauert's example shows that we cannot replace in the previous statement the projectif space by a 
 torus.
 
 2) Similarly, if $K$ is a minimal compact laminated set in $\P^k,$ with finitely many singular points. Then one can construct a fundamental  sequence  $(U_n)$ of open neighborhoods  of $K$, such that every psh function in $U_n$ is constant. For basics on laminations see for example the survey \cite {FS2} .

 3)  The complement   of the  support  of a  positive  current of bidimension $(1,1)$ and $\ddbar$-closed  is
 $1$-pseudoconvex (in dimension $2$ it is  pseudoconvex). So the support is  quite large see  \cite{FS,S2}.
 The support $H,$ of a positive $\ddbar$-closed current satisfies the local maximum principle for
 psh functions near $H,$ see \cite{S1}.
 
 4) If $U$ is  not  Stein but  admits   a  continuous  psh exhaustion  function $\varphi$, the  classes
 $$C_a:=\left\lbrace  z:\ u(z)=u(a)\quad\text{for every}\quad u\quad\text{psh  continuous}\right\rbrace$$  
 are nontrivial and indeed are on $(\varphi=t),$ hence  some sets $C_a$ are of Hausdorff  dimension 
 larger or equal to $2.$ Finally, the   extremal   Levi currents in $U,$ are Levi currents on $(\varphi=t)$ 
 except  the  level set is  not necessarily   smooth.
 \end{remark}
  We next address the relation with pseudoconcave  manifolds.
  We first recall some definitions, see  \cite{G3} 
  \begin{definition}
A real $\mathcal C^2$ function $u$ defined in a complex manifold $U$ is strictly $q$-convex if the complex Hessian (Levi form) has at least $q$ strictly positive eigenvalues at every point. A complex manifold $U$ is $q$-complete if it admits a strictly $q$-convex exhaustion function.
\end{definition}
\begin{definition}
A function $\rho$ is strictly $q$-convex with corners on $U$ if for every point $p\in U$ there is a neighborhood $U_p$ and  finitely many strictly $q$-convex $\mathcal C^2$ functions $\{\rho_{p,j}\}_{j \leq \ell_p}$
on $U_p$ such that $\rho_{|U_p}= \max_{j\leq \ell_p} \{\rho_{p,j}\}.$
The manifold $U$ is strictly $q$-complete with corners if it admits an exhaustion function which is strictly $q$-convex with corners.
\end{definition}
\begin{theorem}\label{T:4.4}
 Let $K$ be a compact set in a connected complex manifold $U.$    Assume $U\setminus K,$  is strictly $2$-complete with corners.Then $U\setminus K,$  admits no non-constant psh function. 
 Moreover there is a  non-zero, positive $\ddbar$-closed current $T, $  of  bidimension $(1,1)$ supported in $K.$
  \end{theorem}

\proof
Let $\rho$ denote the strictly $2$-convex exhaustion  function with corners. For each $p$ in a level set
$(\rho=c)$, there is  on a neighborhood $U_p,$ a strictly 2-convex function $\rho_{p}$,   such that
 $\rho_{p}(p)=c,$ and $\rho \geq \rho_{p}$ in $U_p.$ If at $p$ the gradient of $\rho_{p}$ is non zero there is still a strictly positive eigenvalue of the Levi-form in the tangent space. Using the Taylor expansion at $p$ one sees easily 
that , there is a holomorphic disc through $p$ and otherwise contained in $(\rho \geq c),$ which enters in $(\rho>c),$ see
  \cite{Ho} p.51. If the gradient of $\rho_{p}$ vanishes at $p,$ the construction of an analytic disc with the above properties is even simpler.  Hence if $u$ is psh in $(\rho>c),$ by maximum principle, $u$ is constant on each component of $(\rho>c).$ Since $c$ is arbitrary and $U$ is connected the result follows. 
 
 In particular there are no strictly psh functions near $K$. The existence of $T$ follows from the duality principle we have already used.
 \endproof
 We end up this section with few remarks on bounded psh functions.
  
 For a non-compact connected Riemann Surface $N,$ the existence of a non-constant  bounded subharmonic function is equivalent to the existence of a Green function with a pole at a point $p$ in $N. $ The notion seems much less explored in several complex variables. We give few remarks on the question. We introduce first the following definition. 
  \begin{definition}
 A connected complex manifold $N$ is Ahlfors hyperbolic iff it admits a smooth bounded 
 strictly psh function.
 \end{definition}
 Using Richberg's approximation Theorem  \cite{R}, this is equivalent to the existence of a continuous
 bounded strictly psh function. It is clear that such manifold does not have non-zero Liouville currents.
 We give a class of examples.
 
 Let $\P^k$ denote the complex projective space of dimension $k$. Consider an endomorphism
$f:\P^k\rightarrow\P^k$ which is holomorphic and of algebraic degree $d$ strictly larger than $1.$
Let $\omega$ denote the Fubini-Study form on $\P^k.$  The Green current $T$ associated to $f$ is given by:
$T= lim (d^{-n} (f^n)^*(\omega))= \omega + i\ddbar g.$ The function $g$ is H\"older continuous.

Moreover the complement of $supp(T)$ is the Fatou set \cite {DS2}. So any component $U$ of the Fatou set is Ahlfors hyperbolic.

More generally, let $ (M, \omega )$ be a compact   K\"ahler manifold. Let $T$ be a positive closed current of bidegree $(1,1)$ cohomologous to  $\omega$, with locally bounded potentials. Then the components of the complement of $supp(T),$ are Ahlfors hyperbolic.

 Let $a$ be a positive irrational number. Consider in $ \C^2$ the domain
 $$U_a:=\left\lbrace (z,w) / |w||z|^a <1 \right\rbrace.$$ This domain is Stein, but is not an Ahlfors 
 hyperbolic domain. 
 
 Indeed any bounded psh function is constant on the level sets of the function $u(z,w)= |w||z|^a.$ It is easy to see that each such level sets supports a non zero positive closed Liouville current $T$. The current is unique up to a multiplicative constant. On the level set $(u=c),$ the current is given by,
 $T_c= i\ddbar (log (sup(u,c))).$
 
  In the definition we have asked for the functions to be smooth to avoid examples like the following.
  Let Let $\varphi$ be a subharmonic function in $\C,$  taking the value  $-\infty,$ on a dense set.
  Define $$U:=\left\lbrace (z,w) / |w|exp(\varphi(z)) <1 \right\rbrace.$$ 
   Then $U$ admits non-constant bounded psh functions, but every continuous bounded psh function is constant.
 
 In  \cite{OS} there is an example of a Stein  domain $U$ with smooth boundary, relatively compact in a homogeneous manifold $M,$ such that all bounded psh functions in $U$ are constant. Indeed $U$is 
 foliated by images of $\C$ which cluster on the boundary.
 
 \begin{proposition}
  Let $(M,\omega),$ be a compact K\"ahler manifold. Let $U \subset M,$ be a domain, with a non-constant continuous bounded psh function $u$ defined in $U, $ reaching it's minimum $c,$ in $U.$ Let $X_c:= (z\in U, u(z)= c).$
  
   Either $\overline{X_c}$ supports a non-zero positive $\ddbar$- closed current, or there is in $U,$ a bounded continuous psh function $v,$ which is strictly psh in a neighborhood of
  $X_c.$
  
 \end{proposition}
\proof
Suppose there is no non-zero positive $\ddbar$-closed current, supported on $X_c.$ Then, there is a strictly psh function $w,$ in a neighborhood $W$ of $\overline{X_c}.$ We can assume that $c=0,$ and that  on $W,$ $0<w<1.$  Composing with a convex, increasing function, we can assume that out of $W$ we have $u>1.$ It suffices to define $v= sup(u,w).$ The function is well defined in $U,$ and is strictly psh near $X_c.$
 \endproof


\section{Manifolds with holomorphic vector fields}\label{S:manifolds_vector_fields}


In this    section we discuss   the Levi problem on manifolds $M,$ on which
 the  space  $\mathcal V$ of holomorphic vector fields is of positive  dimension.  
Hirschowitz has  considered  manifolds $M$ on which  at every point $p,$ $\mathcal V$ generates the  tangent space  of $M$ at $p.$
 He called  such manifolds
infinitesimally homogeneous. When $\dim \mathcal V\geq 1,$ we   will  say that
$M$ is  partially infinitesimally homogeneous. 

 Let $D$ denote  the open unit disc in $\C.$
Recall that   a domain  $U$ in a complex manifold $M$ satisfies the Kontinuit\"atssatz if the following  holds. For any sequence  $f_j:\  \overline{D}\to U$   of holomorphic  maps 
on $D,$  continuous on $\overline{D},$ such that  $\bigcup_j  f_j(\partial D)$ is relatively compact  in $U,$
then $\bigcup_j  f_j(\overline D)$ is relatively compact  in $U.$

Hirschowitz showed that    if $U$ satisfies   the Kontinuit\"atssatz  in an infinitesimally homogeneous
manifold, then $U$ admits  a  continuous psh exhaustion function. We first  refine his result, then  we describe  the pseudoconvex  non-Stein  domains in an infinitesimally homogeneous manifold. They   carry
a holomorphic  foliation  with very special  dynamics.

\begin{theorem}\label{T:5.1}
Let $U\subset M$   be  a domain   satisfying the  Kontinuit\"atssatz. Assume that  at  each point  $p\in\partial U,$
there is  $Z\in \mathcal V$  transverse  to the  boundary  at $p,$ we will write  $\partial U\pitchfork\mathcal V.$ Then $U$ admits a 
psh exhaustion  function $\varphi.$
\end{theorem}
\proof
A vector field $Z$ is  transverse  at $p$ to $\partial U,$ if   the local   solution of $Z$  around $p$
passes through  $U$ and  through $M\setminus \overline{U}.$ For a vector field $Z$ in $M,$ we consider
the flow $g_Z(z,\zeta),$ such that $g_Z(z,0)=z.$ 
Let $ \Omega_z,$ denote the connected component in $\C$ containing $0,$ of the open set
$(\zeta\in\C,\quad  g_Z(z,\zeta) \in U).$ Let  
$$U_Z:=\left\lbrace (z,\zeta):\ z\in M,\zeta\in\C, \zeta \in 
\Omega_z \right\rbrace.
$$
We define  $d_Z(z)$ as   the distance  of $z$  to the  boundary  along  the  vector field $Z.$  More  precisely,
$$
d_Z(z):=\sup\{ |\zeta|: (z,\zeta)\in U_Z\}.
$$
\begin{lemma}
\label{L:5.2}
If $U$  is  not  invariant under the flow $\varphi_Z,$ then  $-\log d_Z$ is  psh on $U.$
\end{lemma}
\begin{proof}
We  observe that if  the  domain $U_Z$ satisfies  the Kontinuit\"atssatz, it will follow that  $-\log$ distance
to the complement  in the  $\zeta$-direction is psh.
When  $Z$ is  transverse  to  the  boundary  at some point,  the function is  not  identically  $-\infty.$

Let $f_j:\  \overline{D}\to U_Z.$ Assume    $\bigcup_j  f_j(\partial D)\Subset U_Z.$
Let $\pi:\   U_Z\mapsto U$ be the projection.  Since $U$ satisfies  the Kontinuit\"atssatz, 
$\bigcup_j  \pi\circ f_j(\overline{ D})\Subset U.$ That   we   have also the compactness  in the $\zeta$-direction  follows from    the  standard results  on solutions of  vector fields.   
\end{proof}
\proof [End of the  proof of Theorem \ref{T:5.1}]
When $U$ is  relatively  compact, we just need  finitely many vector fields  $\mathcal V$, such that
for every point $p\in \partial U,$ there is  $Z\in\mathcal V,$ transverse  to $\partial U$ at $p.$
Then the function   $\sup\{-\log d_Z(z):\ Z\in\mathcal V\}$ is a continuous psh  exhaustion. When $\overline{U}$
is  not compact,  the construction can be  easily  adapted, to get the continuity. Indeed, we need only
finitely many vector fields on each compact of $\overline{U}.$
\endproof

\begin{remark}\label{R:5.3}
\rm
When $M$ is   infinitesimally  homogeneous, the hypothesis is  always satisfied. Otherwise, one  should observe that  if it is    satisfied for $U,$ it holds also  for domains  close enough  
to $U,$ in the $\mathcal C^1$-topology.
\end{remark}
Let $v$  be  a continuous   subharmonic  function  in the unit disc $D.$ We  will say that  $v$ is
strictly subharmonic   at $0$ if the Laplacean $\Delta  v>0$ in a neighborhood of $0.$ This is  equivalent to the  fact that   small $\mathcal C^2$ pertubations of $v$  in a neighborhood of $0,$ are still  subharmonic. For a  continuous  psh  function  $u$ in $U,$ we  will write  $\langle i\ddbar u(z), it_z\wedge \bar{t}_z \rangle=0$ 
iff $u\circ f$ is not  strictly  subharmonic  at $0,$ for a holomorphic map $f:\ D\to  U,$  $f(0)=z,$ $f'(0)=t_z.$
Let $\mathcal P(U)$  denote the  continuous psh functions on $U.$ For a compact set $K\Subset U,$ let $\mathcal P(K) $ denote the cone of  psh  continuous  near $K.$ 
 We define
$$\mathcal N_z(K):=\left\lbrace t_z:\ t_z\in T^{1,0}(U),\  \langle i\ddbar u(z),  it_z\wedge \bar{t}_z \rangle=0\quad\text {for u} \in \mathcal P(K) \right\rbrace. $$
Denote by  $\mathcal N(K):= \bigcup_{z\in U}\mathcal N_z(K).$   
Let  $\mathcal N_z:=\bigcup_{K}\mathcal N_z(K)$ as $K\nearrow U,$ and $\mathcal N:=\bigcup \mathcal N_z.$

\begin{proposition}\label{P:5.4}
Suppose there is $Z\in\mathcal V,$ with $Z(z_0)=t_{z_0}\in\mathcal N_{z_0}.$ Then the orbit of $Z,$ starting  at $z_0$
is  contained in $\mathcal N.$ If moreover  $\partial U\pitchfork\mathcal V,$ then  the orbit   is  complete and   is contained in  a level set  of any psh continuous function   $u\in\mathcal P(U).$ In particular in the level sets of the exhaustion  $\varphi.$

\end{proposition}
\proof
Let $g$ denote the complex flow of $Z.$ Then for $u\in \mathcal  P(K),$ $u\circ g(z_0,\zeta)$ is also in  $\mathcal P(K), $ if $\zeta$ is  small  enough. Since  $g$ is  a local biholomorphism,  $g(z_0,\zeta)(t_{z_0})\in\mathcal 
N_{g(z_0,\zeta)}.$

We can approximate functions in  $\mathcal P(K) $ by  functions in  $\mathcal P(K),$ smooth along
the orbits of $\mathcal V.$ Let $B$ be a neighborhood of $0$ in $\mathcal V$ and let $D$ denote the unit disc in $\C.$  Let  $\rho$  be an approximation of the identity in $B\times D.$ It suffices to consider the approximation
 $$\langle \rho(Z, \zeta), u (g_Z(z, \zeta) \rangle.$$
We can use functions smooth on orbits. If $Z\in\mathcal N,$ we get that  $\langle  \dbar u,Z\rangle=0.$
Indeed, we can apply the definition to $exp(u).$ 
So $u$ is constant along the orbit of $Z.$
Since $\varphi$  is  an exhaustion, the  orbit  is  contained in  a  compact level set  of $\varphi$ and hence  we have  a  holomorphic image of $\C$  in that level set.
\endproof

\begin{theorem}\label{T:5.6}  Suppose   $M$ is  infinitesimally homogeneous. Let  $U\subset M$
satisfy the Kontinuit\"atssatz. There is  an integer   $0\leq  d<n,$   and   a foliation $\Fc,$
with leaves  of dimension $d$ on the level sets of   the  exhaustion $\varphi.$ If  $d=0,$ then $U$ is  Stein.
If $d\geq 1,$  then $(i\ddbar \varphi)^{n+d-1}=0$   and  there is   a  positive closed   Levi  current  $T_t,$  of bidimension $(d,d)$  and  mass one, on each  level set  $(\varphi=t),$ for $t\geq   t_0.$ 

 When $U\Subset M$
and $U$ is  non Stein  there is  a  positive closed Levi  current  $S$  of bidimension $(d,d)$ on  $\partial U.$ 
\end{theorem}
\begin{lemma}\label{L:5.7}
 The  bundle $\mathcal N,$ is  of rank  $d,$ with $0\leq  d<n.$ The   sections  are stable  under Lie bracket. 
\end{lemma}
\proof
Since $M$ is  infinitesimally  homogeneous, for each  non-zero vector $t_z\in\mathcal N_z$  there is a   holomorphic   vector field $Z$ in that  direction. So we can apply   Proposition \ref{P:5.4}.  Moreover,  the image  by the  flow $g$ of  $\mathcal N$ is  contained in $\mathcal N.$  So the   support of $\mathcal N$ is $U.$  We show that  $\mathcal N$ is  stable  under   Lie bracket. 

In an  infinitely  homogeneous manifold,  continuous psh functions near $K$ are approximable by smooth ones \cite {H2} . So to analyze $\mathcal N_K,$
we can use smooth psh  functions, near a compact set $K.$

Let $X,$ $Y$ and $Z$  be $(1,0)$-holomorphic vector fields.  If
$X,Y\in\mathcal N(K),$ then  $\langle  \dbar u,X\rangle=\langle \dbar u, Y\rangle=0$ and 
$\langle \ddbar u, X\wedge \overline{Z} \rangle =0$ for every   $Z\in\mathcal V,$ and for any smooth psh function in a neighborhood of $K.$ We also have for type reasons:
$$
\langle \ddbar u, [X,Y]\wedge\overline{Z}\rangle=-\langle \partial u,  [[X,Y],\overline{Z}] \rangle.\quad\
$$
Jacobi's identity  gives that 
$$
 [[X,Y],\overline{Z}] = [X,[Y,\overline{Z}]] -[Y,[X,\overline{Z}]].
$$
So if  $X,Y\in\mathcal N(K),$ i.e.  $\langle  \ddbar u, X\wedge\overline{X}\rangle=\langle\ddbar u, Y\wedge\overline{Y}\rangle=0,$ then 

$$\langle  \ddbar u, X\wedge\overline{Z}\rangle=\langle\ddbar u, Y\wedge\overline{Z}\rangle=0$$
for every $Z\in\mathcal V.$ It follows that $[X,Y]\in \mathcal N(K).$ t is clear that the rank $d_z$ of $\mathcal N_z(K)$ increases as $K$ increases. Hence it stabilizes.
Observe tha $d_z$ cannot drop since strict plurisubharmonicity is an open condition.
So  $\mathcal N$ is a  bundle of dimension $d,$ stable  under   Lie bracket.
 
\endproof

\proof[Proof of  Theorem  \ref{T:5.6}]
To the bundle   $\mathcal N$, we associate  a  foliation $\mathcal F$. Moreover for every $Z \in \mathcal N$ we have that 
$\langle \partial u, Z \rangle =0. $ Hence the leaves are contained in the level sets of functions in 
$\mathcal P(U).$ In particular the  leaves are contained in $(\varphi=\const).$ 

Since  there is an  exhaustion  function, necessarily  $d<n.$ When $d=0,$ it is  easy to construct a  strictly  psh exhaustion function

Any continuous psh function $u,$  is constant on the leaves, hence  $(i\ddbar u)^{n-d+1}=0.$  In particular, $\varphi$  is  constant  on leaves, hence $(i\ddbar \varphi)^{n-d+1}=0.$ We can replace
$\varphi$ by  $\exp(\varphi),$  so it follows that   $i\partial \varphi\wedge\dbar \varphi\wedge (i\ddbar\varphi)^{n-d}=0.$  Consequently, for every nonnegative  function  $\chi,$ the current
$\chi(\varphi)(i\ddbar\varphi)^{n-d}$ is  closed.    Let  $c$ denote it's mass.
We can  construct   a  positive closed  current   $T_{t_0}$ of  mass $1$  on $(\varphi=t_0)$
as  a limit   of ${1\over c_j} \chi_j(\varphi)(i\ddbar\varphi)^{n-d}.$
It follows also that   when $d>0$  there 
is a positive  closed current $S$  of bidimension $(d,d)$ and mass $1$ on $\partial U.$  It suffices to take a cluster point of the currents $T_t.$

\endproof
\begin{remark}\label{R:5.8} \rm
   If $M$ is K\"ahler, then
 $S$ is nef, i.e., it is  a  limit of smooth  strictly positive closed forms.
 \end{remark}
 \begin{corollary}
 Let $(M,\omega)$ be an infinitesimally  homogeneous   compact K\"ahler manifold. Let $U\subset M$
 be  a  domain  satisfying the  Kontinuit\"atssatz. If  $H^2_{\text{dR}}(U)=0,$ then  $U$ is  Stein.
\end{corollary}
\proof Let $T$ be a positive closed  current   of bidimension $(1,1)$ with compact support in $U.$
Let $\omega$ be a K\"ahler form. The cohomological hypothesis on the de Rham group, implies
that  we can write $\omega=d\alpha,$ with $\alpha$ smooth.
Then 
$$
\langle T,\omega\rangle=\langle T,d\alpha\rangle=-\langle dT,\alpha\rangle=0.
$$
So the dimension of the bundle $\mathcal N$ is $d=0$ and   hence $U$ is  Stein.
\endproof
\begin{remark}\label{R:5.9} \rm
Assume $U\Subset M,$ with $(M,\omega),$ compact  K\"ahler, not infinitesimally homogeneous. Suppose $H^2_{\text{dR}}(U)=0,$
and that $U$ admits a continuous psh exhaustion. If $U,$ is not Stein, then there are non-zero non-closed but $i\ddbar$-closed currents with compact support in $U.$ This follows from the argument in the previous Corollary  and  from Theorem \ref{T:main_3}.
  \end{remark}

\begin{corollary}\label{C:5.10}
Supppose $(M,U)$ are as in  Theorem \ref{T:5.6}.  Assume that  $M$ is a  compact K\"ahler manifold.  
Suppose also that $U$ does not contain a compact variety of dimension $d.$
Let  $T_t$ be the positive closed current  constructed in Theorem \ref{T:5.6}.
Then $\{ T_t\}^2=0.$
If $d=n-1,$ all  the   currents $T_t$  are  nef  and are in the  same cohomology class.    
\end{corollary}
\proof
The  currents $T_t$  are directed  by a foliation without singularities and they are closed and diffuse.  Then a  result of Kaufman  \cite{K} states that  the cohomology class  $\{ T_t\}$ of $T_t$ satisfies
  $\{ T_t\}^2=0$ provided $2d\geq n.$

When $d=n-1,$ $\{ T_t\}^2=\{ T_{t'}\}^2=0.$ If  $t\not=t'$, since the supports of  $T_t$ and $T_t'$ are
disjoint  we get that  $\{ T_t\}\smile \{ T_{t'}\}=0.$  As we have seen, $\{ T_t\}$ and $\{ T_{t'}\}$ are nef, then  the Hodge-Riemann signature Theorem,
 implies that  $\{ T_t\}$ and $\{ T_{t'}\}$ are proportional. If the  two currents are on the  same level set,
 we use a  current  on another level set.
\endproof

\begin{remark}\label{R:5.11}\rm
Theorem  \ref{T:5.6} can be improved  as  follows.  Suppose $\mathcal V$ generates $T^{1,0}M$  at one  point   $z_0.$
Let $A$ be   the analytic set where  $\rank  \mathcal V(z)\leq n-1.$ Assume   $U\pitchfork\mathcal V.$
Then there is a continuous psh exhaustion $\varphi.$
There is  also, an integer $d,$ $0\leq d<n$ and a   foliation   with leaves  of dimension $d$ on each $(\varphi=t)\setminus A.$ In particular,  if $d=0$ and  $A$ is  Stein,  then $U$ is  Stein.  If $U$ is not Stein, then there is  a nontrivial  holomorphic  image of $\C,$ which is relatively compact  in $U.$ One shows if $d=0$ that   any positive $\ddbar$-closed  current has  no mass out of $A.$ So if  $A$ is  Stein, it
has  also no mass on $A.$  Hence, $U$ is  Stein by Theorem \ref{T:main_3}.

\end{remark}
\begin{remark}\rm Let $U\Subset M$ be a  pseudoconvex domain with smooth boundary  in  an infinitesimally
homogeneous manifold. Assume   it is  not Stein and  let $d$ be the  dimension of the  leaves of the  associated foliation $\mathcal F$. Then $\mathcal F$  extends  to a foliation on $\partial U,$ with leaves  of dimension  $d.$
This  is a  case  where   the  dimension  of the leaves does not change.
\end{remark}

\begin{corollary}\label{C:5.12}
Let $U\subset M$ be  as in Theorem \ref{T:5.6}. Assume  $\partial U$
is  smooth. Assume that   at  a point $p\in\partial U$  the  rank of the Levi form is maximal and  equal to $n-l.$ 
Then the dimension  $d$ of the foliation  satisfies, $ d\leq l-1.$  In particular, if $p$  is  a point   of strict pseudoconvexity, then $U$ is  Stein.
 \end{corollary}
 \proof
 Suppose $p$ is  a point of strict pseudoconvexity i.e $l=1.$ Then there is  a continuous psh function $\psi$ in $U:$
 $\psi(p)=0,$  $\psi<0$ on   $B(p,r)\cap U$ and $\psi$   strictly psh near $p.$ Such a function
 implies that $d=0$ and  hence $U$ is  Stein. For the general case     we can  cut locally  near $p,$ by a subspace $L_\alpha$   of dimension $n-l+1$ such that $L_\alpha\cap\partial U$  is  strictly  pseudoconvex 
 at $p.$ Then   the foliation   induces on $L_\alpha,$ leaves of dimension $0.$ Hence, the dimension of the  original  foliation  satifies  $ d\leq l-1.$ 
\endproof
I thank Masanori Adachi for pointing out a slip in the previous formulation of the above statement.
\begin{remark}
\rm
When  $M$ is compact homogeneous and $U$  has  a  point of strict  pseudoconvexity, the  result is  due to Michel  \cite{M}.

\end{remark}

\section{Levi-problem for Pfaff systems}\label{S:Levi-Pfaff}


Let $(M,\omega)$  be  a complex hermitian  manifold. Fix  $\mathcal S=(\alpha_j)_{j\leq m}$ a Pfaff system, i.e.,
the  $\alpha_j$ are $(1,0)$-forms of class $\mathcal C^1$ on $M.$ Assume that for every
$z\in M,$  $\bigcap_{j\leq m} \ker \alpha_j(z)\not=\{0\}.$  We will say that  $M$ is $\mathcal S$-strongly
pseudoconvex if it admits  a smooth  exhaustion  $u$ such that outside of a compact set in $M,$
\begin{equation}\label{e:6.1}
\langle  i\ddbar u(z), it_z\wedge \bar{t}_z \rangle>0\quad\text{for}\quad  t_z\not=0,\quad \langle\alpha_j(z),t_z\rangle=0,\quad  1\leq j\leq m.  
\end{equation}
Let $U\Subset M$    be  a domain    with smooth  boundary and  defining function $r.$  We will  say that    $U$ is
$\mathcal S$-pseudoconvex  if

\begin{equation}\label{e:6.2}
\langle  i\ddbar u(z), it_z\wedge \bar{t}_z \rangle\geq 0\quad\text{when}\quad  \langle  \partial r(z), t_z\rangle=\langle\alpha_j(z),t_z\rangle=0,\quad  1\leq j\leq m.  
\end{equation}
  The  basic example of this  situation    is     when the  $(\alpha_j)$   are holomorphic  and the  system is  integrable. Then we get a notion of uniform pseudoconvexity on leaves.
  
  The  question we address is   assuming that  $U$ is $\mathcal S$-pseudoconvex in a complex manifold $M,$ under which  conditions is $U$ $\mathcal S$-strongly
pseudoconvex?

It is  natural to introduce  the cone  $\mathcal P_S$ of $\mathcal C^2$-smooth functions in $U$  such that 
$\langle i\ddbar v, it\wedge  \bar{t}\rangle\geq 0$  when $\langle \alpha_j,t\rangle=0$  for $1\leq j\leq m,$
see \cite{S1}.  Denote  by   $\overline{\mathcal P}_S$   the space of continuous  functions   which are  decreasing limits
of functions   in $\mathcal P_S$ 

We can extend  some of the results  from previous  sections to this  context. As observed in  \cite{S1},
the estimates \eqref{e:4.2} and \eqref{e:4.3} are valid for positive $\ddbar-$ closed currents directed by $\mathcal S,$ and for functions in $\overline{\mathcal P}_S.$ 
A  Levi current  $T$  for the  system $\mathcal S$ on $U$ is a positive current of bidimension $(1,1)$ such that
$$
i\ddbar T=0,\quad T\wedge \alpha_j=0,\ \ 1\leq j\leq m,  \quad   T\wedge i\ddbar u=0\quad \text{for every}\quad
u\in\mathcal P_S.
$$
A Levi current   for the  system $\mathcal S$  on $(r=0)$ satisfies
\begin{equation}\label{e:6.3}
i\ddbar T=0,\quad T\wedge i\ddbar r=0,\quad T\wedge \partial r=0,\quad  T\wedge \alpha_j=0,\ \ 1\leq j\leq m,\quad  \langle T,\omega\rangle=1.
\end{equation}
We just state  some extensions, leaving the proof to the reader.
\begin{theorem}\label{T:6.1}
Let $U\Subset M$ be  an  $\mathcal S$-pseudoconvex domain  with smooth  boundary. If there is  no Levi current for the  system   $\mathcal S$ on $\partial U,$ then  $U$ is  strongly $\mathcal S$- pseudoconvex. Moreover,   there is  a smooth function $v,$   such that for $A$ positive  large enough and $\eta>0$ small enough,
$\hat\rho:=-(-re^{-Av})^\eta$ satisfies $i\ddbar \hat\rho\gtrsim C\omega|\hat\rho|  $ on vectors such that
$\langle   \alpha_j(z),t_z\rangle=0,$  $1\leq j\leq  m.$
\end{theorem}
\proof[Sketch of  proof]
One shows that    if $T\geq 0,$   $T\wedge \alpha_j=0$ for $ 1\leq j\leq m, $ and $i\ddbar T=0,$
then  $T\wedge\partial r=0$ and  $T\wedge i\ddbar r=0.$  So a duality argument  implies that   if  there is   no  Levi current for the  system $\mathcal S$  on $\partial U,$ there is   a  function $v\in\mathcal P$ which is   strictly $\mathcal S$-psh, i.e., $\langle  i\ddbar v(z),t_z\wedge \bar{t}_z\rangle >0$ when   $\langle   \alpha_j,t_z\rangle=0$ for $1\leq j\leq m.$  The proof then follows the lines  of the proof of  Theorem
\ref{T:2.2} using relations  like
$$
\langle   i\ddbar r(z),t_z\wedge \bar{t}_z\rangle \geq   -C|\langle \partial r(z),t\rangle|  |t| -C\sum_{j=1}^m|\langle \alpha_j(z),t\rangle||t|
$$
for $z$ in a neighborhood  of $\partial U.$
\endproof
\begin{remark}\rm
In  \cite {BerndtssonSibony} and  \cite{S1},  H\"ormander  type  estimates, for  $\dbar$ with respect to $\mathcal S$-directed currents  with respect  to $\mathcal S$-pseudoconvex   functions
as  weights, are given.
\end{remark}
\begin{theorem}\label{T:6.2}
Suppose  $U\subset M$  admits a continuous   exhaustion  function which  is $\mathcal S$-pseudoconvex. Then  $U$
admits  a  strictly $\mathcal S$-pseudoconvex exhaustion  if and only if    there  is  no Levi current, for
the system $ \mathcal S,$  with compact support in  $U.$ 
\end{theorem}
The proof is  an   adaptation   of  the proof  in Section  \ref{S:Levi}. We omit it.


\noindent
Nessim Sibony, Universit{\'e} Paris-Sud,\\
\noindent
and Korea  Institute For Advanced Studies, Seoul \\
\noindent
{\tt Nessim.Sibony@math.u-psud.fr},

\end{document}